\documentclass[11pt]{amsart}

\usepackage{times}

\usepackage[left=1in,right=1in,top=1.5in,bottom=1.5in]{geometry}
\usepackage{amssymb}
\usepackage{latexsym, amsmath, amscd,amsthm}
\usepackage{graphicx}
\usepackage[percent]{overpic}
\usepackage{units}
\usepackage{hyperref}
\usepackage{color}
\usepackage{xcolor}

\usepackage{diagrams}

\newtheorem{theorem}{Theorem}

\newtheorem{lemma}[theorem]{Lemma}
\newtheorem{proposition}[theorem]{Proposition}
\newtheorem{corollary}[theorem]{Corollary}

\theoremstyle{definition}
\newtheorem{definition}[theorem]{Definition}

\newcommand{\R}{\mathbb{R}}
\newcommand{\Z}{\mathbb{Z}}

\newcommand{\so}{\operatorname{SO}(3)}
\newcommand{\Orth}{\operatorname{O}(3)}

\newcommand{\Thi}{\operatorname{Thi}}
\newcommand{\Strut}{\operatorname{Strut}}
\newcommand{\Kink}{\operatorname{Kink}}
\newcommand{\cross}{\times}
\newcommand{\reach}{\operatorname{reach}}

\newcommand\Len{\operatorname{Len}}
\newcommand{\sot}{\operatorname{SO}(2)}
\newcommand{\ot}{\operatorname{O}(2)}
\newcommand\Rop{\operatorname{Rop}}
\def\co{\colon\!}
\newcommand{\norm}[1]{\left\| #1\right \|}
\newcommand{\pd}[2]{\operatorname{pd}(#1,#2)} 
\newcommand{\psit}[2]{\psi(#1,#2)}  

\newcommand{\CL}{\operatorname{Osc}L}
\newcommand{\CLb}{\overline{\CL}}
\newcommand{\Sym}{\operatorname{Sym}}
\newcommand{\MCG}{\operatorname{MCG}}

\newcommand{\gCL}{\operatorname{Osc}gL}
\newcommand{\gCLb}{\overline{\gCL}}

\newcommand{\imL}{\operatorname{Image}(L)}

\newcommand{\Isom}{\operatorname{Isom}(\R^3)}

\newcommand{\firstvariationofR}{Lemma~4.2 }
\newcommand{\firstvariationofThi}{Theorem~4.5 }
\newcommand{\firstvariationofThisuperlinear}{Corollary~4.6 }
\newcommand{\ropelengthcriticality}{Definition~4.8 }
\newcommand{\minimizersexist}{Corollary~3.20 }
\newcommand{\thicknessuppersemicontinuous}{Proposition~3.19 }

\begin{document}
\title
     {Symmetric Criticality for Tight Knots}
     
\author{Jason Cantarella}
\address{Department of Mathematics,
University of Georgia,
Athens, GA 30602}
\email{jason@math.uga.edu}

\author{Jennifer Ellis}
\address{Department of Mathematics,
Mercer University,
Macon, GA}
\email{jenn.c.ellis@gmail.com}

\author{Joseph H.G. Fu}
\address{Department of Mathematics,
University of Georgia,
Athens, GA 30602}
\email{fu@math.uga.edu}

\author{Matt Mastin}
\address{Athens, GA}
\email{matt.mastin@gmail.com}


\begin{abstract} 
We prove a version of symmetric criticality for ropelength-critical knots. Our theorem implies that a knot or link with a symmetric representative has a ropelength-critical configuration with the same symmetry. We use this to construct new examples of ropelength critical configurations for knots and links which are different from the ropelength minima for these knot and link types.
\end{abstract}
\date{\today}

\maketitle
\section{Introduction}\label{sect:intro} 

Our goal in this paper is to investigate the shapes of knotted tubes of fixed radius and circular cross section with minimal length. This is known as the \emph{ropelength problem}: minimize the length of a knotted or linked curve $L$ subject to the constraint that the curve has ``thickness one''. There are many equivalent characterizations of the unit thickness constraint\cite{MR2000c:57008,MR99k:57025,Cantarella:2011vq}, but the simplest uses Federer's concept of ``reach''~\cite{MR22:961}: a curve has unit thickness if every point in a unit-radius neighborhood of the curve has a unique nearest neighbor on the curve. In previous work, some of us (Cantarella and Fu) developed a criticality theory for curves with a thickness constraint~\cite{Cantarella:2006by, Cantarella:2011vq}. We gave a kind of Euler-Lagrange equation for curves which were min-critical for length when constrained by thickness which allowed us to characterize some families of critical curves explicitly and to solve explicitly for critical representatives of other link types (such as the Borromean rings). 

An old idea in analyzing geometric functionals is the principle of \emph{symmetric criticality} as in the classic paper of Palais~\cite{Palais:1979ud}. The principle states that when a group $G$ acts on a smooth manifold $M$ by diffeomorphisms, fixing a submanifold $\Sigma$ of $M$ and a smooth function $f \co M \rightarrow \R$, we expect that $p \in \Sigma$ will be critical for $f$ on $M$ if and only if $p$ is critical for $f$ on $\Sigma$. That is, we need only consider the effect of \emph{symmetric} variations of $p$ on $f$ to decide whether all variations of $p$ have directional derivative of $f$ equal to zero. Palais gives hypotheses on $G$ and $M$ under which the principle is guaranteed to hold.

The idea of this paper is to prove a similar symmetric criticality principle for tight knots. Although we are dealing with constrained criticality with a one-sided nonsmooth constraint, so that Palais' theorems do not directly apply, the criticality theory from~\cite{Cantarella:2011vq} will prove adaptable to this situation. We will then use our theory to analyze a few interesting examples. We first consider the case of $(p,q)$ torus knots, showing that any such knot has at least two distinct ropelength-critical configurations, one with $p$-fold and one with $q$-fold rotational symmetry.
Coward and Hass~\cite{Coward:2012wh} have recently shown that a split link composed of a connect sum of two nontrivial knots and an unknotted component has a configuration which cannot be isotoped to the global ropelength minimum for the link without strictly increasing the length of the unknotted component. We analyze similar configurations to show that they have mirror-symmetric ropelength critical configurations which are different from their ropelength global minima. Last, we consider an example due to Sakuma~\cite{MR1006701} which we observe must have at least 4 non-congruent $\Z/2\Z$ (2-fold rotationally symmetric) ropelength-critical configurations.

Since Hatcher's proof of the Smale Conjecture~\cite{newkey26}, we have known that the space of unknotted curves is contractible. It has long been conjectured that there is a natural ``energy function'' on unknots whose gradient flow realizes this contraction by evolving every unknot to the round circle~\cite{MR94j:58038}. In particular, such an energy flow would be a natural unknot recognition algorithm. For the links in our examples, we show that the gradient flow of ropelength may encounter critical points before reaching the global minimum energy configuration, and hence that ropelength gradient flow is not a recognition algorithm for these knots and links. 

If these theorems could be extended to unknots, it would rule out ropelength (and, presumably, many other similar energies) as candidates for an unknot recognition flow. There is every reason to believe that there are ropelength critical unknots different from the round circle-- Pieranski et.\ al.~\cite{Pieranski:2001ux} have given convincing numerical evidence, for example. Such configurations are often called ``Gordian'' unknots. We do not know yet whether symmetric criticality can help find Gordian unknots. It follows from work of Freedman and Luo~\cite{Freedman:1995ih} that for any subgroup $G$ of $O(3)$, the space of $G$-invariant unknots is path-connected and contains the round circle. However, Freedman and Luo do not prove that the space of $G$-invariant unknots is contractible. If there was a $G$ so that the space of $G$-invariant unknots had a nonvanishing homotopy group, we could prove the existence of a configuration $L$ of the unknot which was critical among $G$-invariant unknots. Our theorems below would then imply that $L$ was ropelength-critical. 

\section{Thickness and its first variation}

We begin with some definitions.
\begin{definition}[Federer \cite{MR22:961}]
The \textbf{reach} of a closed space curve is the largest $\epsilon$ so that any point in an $\epsilon$-neighborhood of the curve has a unique nearest neighbor on the curve. For space curves, we also refer to reach as half of the \emph{thickness} of the curve.
\end{definition}

It is a classical theorem of Federer that curves with positive thickness are $C^{1,1}$. Given this definition, we can define
\begin{definition}
The \emph{ropelength} $\Rop(L)$ of a curve $L$ in $R^3$ is the ratio of the length of the curve to its thickness. The ropelength $\Rop([L])$ of a link type $[L]$ is the minimum ropelength of all curves in the link type $[L]$.
\end{definition}
It is well known that ropelength minimizing curves or ``tight knots'' exist in every link type (see \cite{MR2003j:57010,MR2000c:57008,Cantarella:2002dx,Cantarella:2011vq} for variants of this theorem). We see in Figure~\ref{fig:thick} that the thickness of a curve $L$ is controlled by both curvature and chordlengths~(cf.~\cite{Cantarella:2011vq}).

\begin{figure}[t]
  \includegraphics{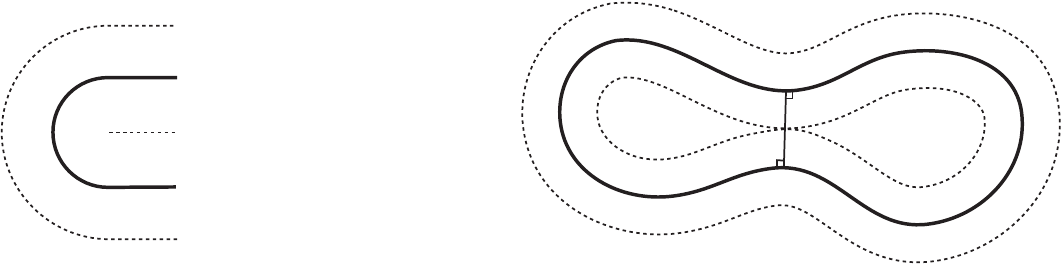}
  \caption{The thickness of a smooth curve $\gamma$ is controlled by
    curvature (as in the left picture), and chordlengths (as in the right picture).}
  \label{fig:thick}
\end{figure}
In order to use our theory below, we need to make this intuition more precise (and deal with some technical issues). Following Figure~\ref{fig:thick}, it would be reasonable to define the ``chordlength'' contribution to thickness in terms of local minima of the self-distance function $d(x,y) = \norm{x - y}$ on $L \cross L$. However, this function has its global minimum of zero along the diagonal (where it is not differentiable), which is not the information we would like to capture. It turns out that nonzero local minima of self-distance are characterized by having the chord $x y$ perpendicular to the tangent vector to $L$ at both endpoints. We can preserve these lengths (but eliminate the problem of the diagonal) with a  technical modification.
\begin{definition}
Let $\psit xy$ be the angle between the chord $xy$ and the normal plane $N_x L$ to $L$ at $x$. Given a link~$L$, the \emph{penalized distance} between two distinct points $x,y\in L$ is
\begin{equation*}
\pd xy := \norm{x - y}\,\sec^2\psit xy.
\end{equation*}
For $y=x$, we set $\pd xx=\infty$. We say that the chord $xy$ is in $\Strut(L)$ if $\pd xy = \Thi(L)$ and $xy$ is in both $N_x L$ and $N_y L$.
\label{def:struts}
\end{definition}

The ``curvature'' contribution to thickness is harder to pin down, since the curves we are dealing with are only $C^{1,1}$. Consider the space $C_3$ of all oriented pointed circles in~$\R^3$,
which we identify with $\R^3 \times TS^2$ by taking $(p,C)$ to correspond
to $(p;T,\kappa) \in  \R^3 \times TS^2$, where $T$ is the oriented unit
tangent to~$C$ at~$p$ and $\kappa$ is its curvature vector there.
Let $R(p,T,\kappa):= \nicefrac 1{|\kappa|} \in (0,\infty]$
be the radius function.
In this formulation the circles $C$ may degenerate to lines,
with $\kappa=0$ and $R=\infty$.
Let $\Pi$ be the projection $\Pi\colon(p,C)\mapsto p$.

Given a $C^{1,1}$ link~$L$, by Rademacher's theorem the set $E$ on which the second derivative
exists has full measure. Thus both the oriented unit tangent vector $T(x)$ and the curvature vector $\kappa(x)$ are well defined for $x \in E$. We let $\CL\subset C_3$ be the set of all
osculating circles:
\begin{equation*}
\CL := \bigl\{\bigl(x,T(x),\kappa(x)\bigr) : x\in E\bigr\}\subset C_3.
\end{equation*}
Its closure $\CLb$ is a compact subset of $C_3$ since $|\kappa|$ is
bounded on~$E$.  Note that $T=T(x)$ for any $(x,T,\kappa)\in\CLb$,
while of course $\kappa\perp T$ is some normal vector; thus we can
view $\CLb$ as a subset of the normal bundle to~$L$.
\begin{definition}
An osculating circle $c \in \CLb$ is said to be a \emph{kink} if its radius $R(c)$ is equal to $\Thi(L)$.
\label{def:kinks}
\end{definition}

We can define a variation $\xi$ of a closed curve $L \in \R^3$ to be the initial velocity field of a smooth family $\phi^t$ for $t \in (-\epsilon,\epsilon)$ of $C^2$ diffeomorphisms of $\R^3 \rightarrow \R^3$ with $\phi^0 = \operatorname{Id}$. The first derivative of the length of a strut as $L$ is carried along by $\phi^t$ is easy to calculate. In~\cite{Cantarella:2011vq}, \firstvariationofR, we show that $\phi^t$ induces a flow on the space of pointed circles $C_3$ and compute the first derivative $\delta_\xi R(c)$ of the radius of a given circle at time $0$ as a function of the variation field $\xi$. Combining these calculations, in \firstvariationofThi of~\cite{Cantarella:2011vq}, we give a formula for the first variation of the thickness of $\gamma$ under $\xi$, assuming that $L$ is $C^{1,1}$ and $\Thi(L) < \infty$:
\begin{align}
\delta_\xi\Thi (L)
   &:= \frac{d \, \Thi(\phi^t(L))}{dt^+}\biggr|_{t=0} \\
   &= \min\biggl(
        \min_{(x,y)\in\Strut(L)}
            \frac{1}{2} \Bigl<\frac {x-y}{|x-y|}, \xi_x-\xi_y\Bigr>,
	    \min_{c\in\Kink(L)} \delta_\xi R(c)
       \biggl).
\label{eqn:variation}
\end{align}
The important thing to notice about this entire apparatus is that all the definitions and functions involved are \emph{geometric}, in the sense that they are preserved by isometries of $\R^3$. Let $g \in \Isom$ be an isometry of $\R^3$. We will refer to the image of $L$ under $g$ as $gL$. We note that $g$ restricts to a map from $L \rightarrow gL$, and that $g$ induces an automorphism of $C_3$ which takes $\CLb$ to $\gCLb$ which we again denote by $g$. Further, this map preserves the radius of the circles. The derivative map $g_{*}$ also takes a variation vector field $\xi$ of $L$  to a corresponding variation $g_{*}\xi$ of $gL$, which is the first derivative of a corresponding family of diffeomorphisms $g \circ \phi^t$ at time $0$. 

We now see that
\begin{proposition}
\label{prop:invariance}
If $g$ is an isometry of $\R^3$ then 
\begin{equation*}
\delta_{g_*\xi} \Thi(gL) = \delta_\xi \Thi(L)
\end{equation*}
and $\Strut(gL) = g(\Strut(L))$, while $\Kink(gL) = g(\Kink(L))$.
\end{proposition}

\begin{proof}
This is an exercise in chasing definitions. Reach (and hence thickness) is clearly preserved by the isometry $g$. Examining Definition~\ref{def:struts}, we see that penalized distance is preserved as well. This means that the strut set of $gL$ is the image of the strut set of $L$. Examining~\eqref{eqn:variation}, this shows that the portion of $\delta_\xi \Thi$ coming from the strut set is preserved by $g$. A similar argument shows that the kink set of $gL$ is the image of the kink set of $L$, keeping in mind that $g$ preserves radius as a map $C_3 \rightarrow C_3$. The only point that might present difficulty is the fact that $\delta_{g_*\xi} R(gC) = \delta_{\xi} R(c)$. This can be argued directly from the fact that $g$ preserves the radius function $R$ or indirectly from the explicit formula for $\delta_\xi R(c)$ in \firstvariationofR of~\cite{Cantarella:2011vq}.
\end{proof}

The other important fact about the first variation of thickness that we'll need appears as \firstvariationofThisuperlinear in~\cite{Cantarella:2011vq}. We reproduce it here:
\begin{corollary}\label{cor:superlinear}
Suppose $L$ is a $C^{1,1}$ curve with $\reach(L)<\infty$.
Then the operator $\xi \mapsto \delta_\xi \Thi(L)$ is superlinear.
That is, for $a \ge 0$ and vector fields $\xi$ and $\eta$, we have
$$
\delta_{a\xi}\Thi(L) = a \delta_\xi \Thi(L), \qquad
\delta_{\xi+ \eta} \Thi(L) \ge  \delta_\xi \Thi(L) + \delta_\eta \Thi(L).
$$
\end{corollary}

\section{Symmetric Links}

We now establish some basic facts about symmetric links. We start by fixing a compact subgroup $G$ of $\Isom$. Notice that this means that $G \subset \Orth$, since any element with a nonzero translation part (such a screw-motion) generates a noncompact subgroup of $\Isom$. 
\begin{definition}
A knot or link $L$ is said to be \emph{(oriented) $G$-invariant} if the action of $G$ on $\R^3$ restricts to an (orientation-preserving) action on $L$. A variational vector field $\xi$ is said to be \emph{$G$-invariant} if it is invariant under the action of $G$ on $\R^3$. 
\end{definition}
We note that a $G$-invariant variational vector field $\xi$ is invariant under $G$ everywhere on $\R^3$. We now need to distinguish between a link and its parametrization. Every $n$-component link $L$ is parametrized by a locally constant-speed parametrization $L \co S^1 \sqcup \cdots \sqcup S^1 \rightarrow \R^3$. This parametrization is unique up to choice of basepoints on the circles (which are all unit circles). 

We can refer to a point in the disjoint union $S^1 \sqcup \cdots \sqcup S^1$ by $(\phi,i)$ where $\phi$ is an angle and $i$ an index denoting which circle the angle is taken on. The isometry group of this disjoint union of circles is the semidirect product $\ot^n \rtimes S_n$, where an element $\gamma = (\gamma_1,\dots,\gamma_n,p)$ acts on a point $(\phi,i)$ by 
\begin{equation*}
\gamma(\phi,i) = (\gamma_i (\phi), p(i))
\end{equation*}
An isometry of $\R^3$ preserves the speed of a parametrization, so $L$ is invariant under $g \in \Isom$ if and only if there is a corresponding $\hat{g} \in \ot^n \rtimes S_n$ so that the diagram below commutes:
\begin{equation*}
\begin{diagram}
 S^1 \sqcup \cdots \sqcup S^1 & \rTo_L & \imL \subset \R^3   \\
\dTo_{\hat{g}}  &  & \dTo_g              \\
 S^1 \sqcup \cdots \sqcup S^1 & \rTo_L & \imL \subset \R^3   \\
\end{diagram}
\end{equation*}
We note that $L$ is \emph{oriented}-invariant under $g$ if and only if this $\hat{g} \in \sot^n \rtimes S_n$.

We now define symmetric link types.
\begin{definition}
Fixing a compact subgroup $G$ of $\sot$, we say that a \emph{symmetric link type} $[L]_G$ is the equivalence class of $G$-invariant curves which are isotopic to $L$ through $G$-invariant curves.
\end{definition}
Symmetric link types are different from topological link types. For instance, we will see that not every topological link type \emph{has} a $G$-invariant representative. Further, even when a topological link type $[L]$ has $G$-invariant representatives $L_0$ and $L_1$, the symmetric link types $[L_0]_G$ and $[L_1]_G$ may be different if every isotopy between $L_0$ and $L_1$ breaks the symmetry. In this case, we will be able to show that each of $[L_0]_G$ and $[L_1]_G$ contains a ropelength-critical configuration in $[L]$ and hence generate more than one ropelength critical configuration in $[L]$. We note that this approach will not work for unknots: it follows from Freedman-Luo~\cite{Freedman:1995ih} that for any $G$, the space of $G$-invariant unknots is connected. Hence this space forms a single symmetric link type whose ropelength minimizer is the round circle.

\subsection{Existence of symmetric minimizers}

We now establish the existence of ropelength minimizing curves in a symmetric link type $[L]_G$. 
\begin{lemma}
If $L_i \rightarrow L$ with respect to the $C^1$ metric, and each $L_i$ is (oriented) $G$-invariant, then $L$ is (oriented) $G$-invariant.
\label{lem:limits are invariant}
\end{lemma}

\begin{proof}
Fix any $g \in G$. Since each $L_i$ is (oriented) invariant under $g$, there are $\hat{g}_i$ in $\ot^n \rtimes S_n$ (or $\sot^n \rtimes S_n$) given by $\hat{g}_i = L_i^{-1} \circ g \circ L_i$. Since $\ot^n \rtimes S_n$ is compact, we may pass to a subsequence and assume $\hat{g_i} \rightarrow \hat{h}$ in $\ot^n \rtimes S_n$. Note that if each $g_i$ was in the compact subspace $\sot^n \rtimes S_n$ of $\ot^n \rtimes S_n$, the limiting $\hat{h}$ is also in $\sot^n \rtimes S_n$.

Since $L_i(\phi,j) \rightarrow L(\phi,j)$ for all $(\phi,j)$, $g \circ L_i(\phi,j) \rightarrow g \circ L(\phi,j)$. Further, $g \circ L_i(\phi,j) = L_i(\hat{g}_i(\phi,j)) \rightarrow L(\hat{h}(\phi,j))$. Thus $g \circ L(\phi,j) = L(\hat{h}(\phi,j))$, showing that $L$ is invariant (or oriented-invariant) under $g$ as desired. 
\end{proof}

\begin{proposition} For any symmetric link type $[L]_G$, if $S$ is the set of curves $L$ in $[L]_G$ with $\Thi(L) \geq 1$, then there exists a length minimizing curve $L_0$ in $S$.
\label{prop:minimizers exist}
\end{proposition}

\begin{proof} 
As in \minimizersexist of \cite{Cantarella:2011vq}, we note that since every $L \in [L]_G$ contains at least one closed curve, its length is bounded below by $\pi$ for all curves with $\Thi \geq 1$ in $[L]$. Since $S$ is nonempty, in view of the Lipschitz conditions on the derivatives a standard application of Arzela-Ascoli yields a $C^1$-convergent sequence of (oriented) $G$-invariant parametrizations $L_i \in S$, converging to some $L_0$ of infimal length. By Lemma~\ref{lem:limits are invariant}, $L_0$ is (oriented) $G$-invariant. Because convergence is $C^1$, the lengths of the $L_i$ converge to the length of $L_0$. By \thicknessuppersemicontinuous in \cite{Cantarella:2011vq}, $\Thi(L_0) \geq \overline{\lim} \Thi(L_i) \geq 1$. Finally, by $C^1$ convergence, $L_0$ is isotopic to all but finitely many of the $L_i$ and in particular, $L_0 \in S \subset [L]$.
\end{proof}

\subsection{$G$-Criticality, regularity, and criticality of local minimizers}

We now recall the definition of thickness-constrained criticality from \cite{Cantarella:2011vq} and extend it to a corresponding definition of $G$-criticality for $G$-invariant links. Our ultimate goal is to prove that the (weaker) hypothesis of $G$-criticality implies the stronger condition of thickness-constrained criticality.
\begin{definition}
A $G$-invariant link $L$ is \emph{$G$-critical} if 
\begin{equation*}
\delta_\xi \Len(L) < 0 \implies \delta_\xi \Thi(L) < 0
\end{equation*}
for all $G$-invariant variation fields $\xi$.  A $G$-invariant link is said to be \emph{strongly $G$-critical} if there exists an $\epsilon > 0$ so that for every $G$-invariant variation field $\xi$ 
\begin{equation*}
\delta_\xi \Len(L) = -1 \implies \delta_\xi \Thi (L) \leq -\epsilon.
\end{equation*}
\label{def:critical}
If $G$ is the subgroup consisting of the identity element, we say that the link is \emph{critical} or \emph{strongly critical}, omitting the group $G$.
\end{definition}
We note that this agrees with the definition of \emph{critical} and \emph{strongly critical} links in \ropelengthcriticality of \cite{Cantarella:2011vq}. An important technical condition is given by the idea of \emph{regularity}.
\begin{definition}
We say $L$ is \emph{regular} if it has a \emph{thickening field}, which is to say a variational vector field $\eta$ on $\R^3$ with $\delta_\eta \Thi(L) > 0$.  We say a $G$-invariant $L$ is $G$-regular if there exists a $G$-invariant.
\end{definition}
Though regularity can fail in certain (somewhat contrived) circumstances if we allow curves with endpoints constrained to lie on submanifolds of $\R^3$, in the setting of this paper, links are always regular and $G$-regular:
\begin{lemma}
Any closed curve $L \subset \R^3$ is regular. If $L$ is $G \subset \Orth$ invariant, then $L$ is $G$-regular as well.
\label{lem:everything is regular}
\end{lemma}
\begin{proof}
Take $\eta$ to be the field generated by scaling, which certainly has a positive first variation for thickness as it increases both lengths and radii of circles to first order. It is $\Orth$-invariant so it is invariant under any $g \in G$ if $L$ is $G$-invariant.
\end{proof}

We now show that regular local minimizers are critical.

\begin{proposition}
If $L$ is $G$-invariant, $G$-regular, and a thickness-constrained local-minimizer for ropelength among $G$-invariant curves, then $L$ is $G$-critical.
\label{prop:minimizers are critical}
\end{proposition}

\begin{proof}
Suppose not. Then there exists a $G$-invariant field $\xi$ with $\delta_{\xi} \Len(L) < 0$ but $\delta_{\xi} \Thi(L) \geq 0$. Since $L$ is $G$-regular, let $\eta$ be a $G$-invariant thickening field. Now the first variation of length is linear in the vector field, so we can choose $c$ small enough that $\delta_{\xi+c\eta} \Len(L) < 0$. By contrast, the first variation of thickness is superlinear in the vector field by Corollary~\ref{cor:superlinear}, so
\begin{equation*} 
\delta_{\xi+c\eta} \Thi(L) \geq  \delta_{\xi} \Thi(L)+\delta_{c\eta} \Thi(L) > 0. 
\end{equation*}
Since $\xi+c\eta$ is $G$-invariant, there are nearby $G$-invariant curves with $\Thi > 1$ but smaller length than $L$, which contradicts the local minimality of $L$. 	
\end{proof}

For our curves, $G$-criticality will imply strong $G$-criticality.
\begin{proposition}
If $L$ is a $G$-invariant curve which is $G$-regular and $G$-critical, then $L$ is strongly $G$-critical.
\label{prop:regular and critical implies strongly critical}
\end{proposition}

\begin{proof}
Let $\eta$ be a $G$-invariant thickening field for $L$. After scaling $\eta$, we may assume that $\delta_{\eta} \Len(L) \leq \frac{1}{2}$. We define $\epsilon := \delta_\eta \Thi(L) > 0$. 

Now take any $G$-invariant $\xi$ with $\delta_{\xi} \Len(L) = -1$. We will show that $\delta_\xi \Thi(L) < -\epsilon$. Again, the first variation of length is linear, so $\delta_{\eta+\xi} \Len(L) \leq -\frac{1}{2}$. By $G$-criticality of $L$, the variation of thickness in the direction $\eta + \xi$ is non-positive. Using superlinearity of the first variation of thickness we have that 
\begin{equation*}
0 \ge \delta_{\eta+\xi} \Thi(L) \geq \delta_{\eta} \Thi(L) + \delta_{\xi} \Thi(L).
\end{equation*}
Thus $\delta_{\xi} \Thi(L) \le -\delta_{\eta} \Thi(L) = -\epsilon$, as desired.
\end{proof}
We now come to the heart of the matter. 
\begin{proposition}
If $L$ is $G$-invariant and strongly $G$-critical then $L$ is strongly critical.
\label{prop:strongly G critical implies strongly critical}
\end{proposition}

\begin{proof}
Since $L$ is strongly $G$-critical, there exists some $\epsilon > 0$ so that for every $G$-invariant $\xi$ with $\delta_\xi \Len(L)  = -1$, we have $\delta_\xi \Thi(L) \leq -\epsilon$. We claim that the same bound holds for \emph{every} $\xi$ with $\delta_\xi \Len(L) = -1$. Suppose not. Then there exists $\xi$ with $\delta_\xi \Len(L) = -1$ and $\delta_\xi \Thi(L) > -\epsilon$. Let $\xi^G$ be the average of the vector fields $g_*\xi$ defined by integration over the compact group $G$. By linearity of $\xi \mapsto \delta_\xi \Len$, 
\begin{equation*}
	\delta_{\xi^G} \Len(L) = -1.
\end{equation*}
Now consulting the first-variation formula for thickness~\eqref{eqn:variation}, if $L$ is invariant under $g \in G$ then $\delta_{g\xi} \Thi(L) = \delta_\xi \Thi(g^{-1} L) = \delta_\xi \Thi(L)$. Thus by superlinearity of $\xi \mapsto \delta_\xi \Thi(L)$,
\begin{equation*}
	\delta_{\xi^G} \Thi(L) \geq \delta_{\xi} \Thi(L) > -\epsilon,
\end{equation*}
which contradicts strong $G$-criticality of $L$.	
\end{proof}	
	
Note that we have used the compactness of $G$ in a really essential way here to define the averaged variation $\xi^G$. However, this is less general than it might seem: $G$ is a subgroup of $\Orth$ and the only compact subgroups of $\Orth$ which are not finite are $\sot \simeq S^1$ and $\ot \simeq S^1 \cross \Z/2\Z$. We can now assemble all of these pieces into our machine for constructing symmetric critical configurations in link types:

\begin{theorem}
There is a strongly critical curve $L$ in every nonempty symmetric link type $[L]_G$. \label{thm:master theorem}
\end{theorem}

\begin{proof}
As before, let $S$ be the subset of $[L]_G$ of curves with thickness at least one. By Proposition~\ref{prop:minimizers exist}, there is a length-minimizing $L_0 \in S$. This curve is regular by Lemma~\ref{lem:everything is regular}, and hence $G$-critical by Proposition~\ref{prop:minimizers are critical}, then strongly $G$-critical by Proposition~\ref{prop:regular and critical implies strongly critical}, and finally strongly critical by Proposition~\ref{prop:strongly G critical implies strongly critical}.
\end{proof}

\section{Examples}

Now that we have our main theorem, it remains to find some interesting symmetric link types. We will rely on the classification of subgroups of $\Orth$ to identify groups $G$ which may appear as symmetry groups. We recall the standard fact that any finite subgroup of $\so$ is either cyclic, dihedral, or one of a finite list of ``point groups'', and that the remaining closed subgroups of $\so$ are copies of $\sot$ and $\ot$ fixing an axis and acting on the perpendicular plane (as well as the full group, of course)~\cite{Wolf:2011tw}. It is a fact that all of these groups may appear as symmetry groups for \emph{links}, but only some may appear as symmetry groups for a \emph{knot}: 

\begin{proposition}
Suppose $L$ is a nontrivial knot which is embedded and $G$-invariant. Then $G$ is trivial, finite dihedral or finite cyclic. If $G = \mathbb{Z}/n\mathbb{Z}$, the classification of subgroups of $\Orth$ means that it is generated by $2\pi/n$-rotation around a fixed axis, and $L$ is disjoint from this axis. If $L$ is oriented $G$-invariant, then $G$ is trivial or finite cyclic.
\label{prop:symknots}
\end{proposition}

This basically a ``folk theorem'' (see page 2 of~\cite{Freedman:1995ih}), so we will not give a detailed proof, but just note that the key idea is to use the fact that elements of $\Orth$ preserve arclength along the knot to map $G$ to $\sot$ and study it there.
We note that if $K$ is a round circle, it is already oriented $\ot$ invariant. This is the maximal symmetry group possible for a space curve, so we cannot hope to conclude that there is an isotopy class of symmetric configurations of the unknot which does not include this configuration. 

\subsection{Symmetries of Torus Knots}

We now consider torus knots. In 1990, Moffatt conjectured~\cite{Moffatt:1990p2427} that 3-fold symmetric and 2-fold symmetric configurations of the trefoil knot should be local minimizers for a magnetic energy closely related to ropelength (cf.\ \cite{Chui:1995et} for more on this conjecture). Sullivan~\cite{newkey93} conducted some numerical experiments with a ``Menger energy'' similar to ropelength where he constructed 3-fold and 2-fold symmetric energy-minimizing configurations with Brakke's~\texttt{Evolver}. (The situation for the conformally invariant M\"obius energy analyzed in~\cite{MR94j:58039} and \cite{MR1702037} is completely different, as the 3-fold and 2-fold symmetric configurations of the trefoil are part of the same family of equivalent conformations.) Our theory is silent on the question of whether any given critical configuration is a local minimum, but can detect the existence of distinct critical configurations of the trefoil with 2-fold and 3-fold symmetries.

We recall that a knot in $\R^3$ is a periodic knot of period $n$ if there is a periodic map $f \co (\R^3,K) \rightarrow (\R^3,K)$ such that $f$ is a $2\pi/n$ rotation around a line $F$ in $\R^3$ and $F \cap K = \emptyset$. It is a classical result~(cf. \cite{Kawauchi:1996uo}) that the only periods of the $(p,q)$ torus knot are the divisors of $p$ and $q$. 

\begin{lemma}
If $K$ is a nontrivial oriented $G$-invariant $(p,q)$ torus knot, then $G$ is either the trivial group or a cyclic group of order $n$ where $n$ divides either $p$ or $q$.
\label{lem:torusknots}
\end{lemma}

\begin{proof}
Since $K$ is a nontrivial knot which is oriented $G$-invariant, from Proposition~\ref{prop:symknots} we know that $G$ is trivial or a finite cyclic group given by rotations around a fixed axis which is disjoint from $K$. Thus $n$ is a period of $K$ and hence $n$ divides either $p$ or $q$ (it cannot divide both, since $\gcd(p,q) = 1$.)
\end{proof}

We can then conclude that 
\begin{theorem}
If $[K]$ is the knot type of a non-trivial $(p,q)$-torus knot, then $[K]$ contains at least two non-congruent ropelength critical configurations; one with $p$-fold and one with $q$-fold rotational symmetry.
\label{thm:torusknots}
\end{theorem}

\begin{proof}
It is easy to construct $p$-fold and $q$-fold (oriented) rotationally symmetric representatives of $K$, so the symmetric knot types $[K]_{\Z/p\Z}$ and $[K]_{\Z/q\Z}$ are nonempty. By Theorem~\ref{thm:master theorem}, each contains a strongly critical configuration in $[K]$. We must prove that these configurations are not congruent. Suppose they are congruent, and without loss of generality that they are the same curve $K_0$. Now $K_0$ has an oriented symmetry group $G$ containing $\Z/p\Z$ and $\Z/q\Z$ as subgroups. By Proposition~\ref{prop:symknots}, $G$ is finite cyclic. Hence $G$ has order at least $pq$, since $\gcd(p,q) = 1$.  But $K_0$ is a $(p,q)$ torus knot, so Lemma~\ref{lem:torusknots} tells us that the order of $G$ divides $p$ or $q$. This contradiction proves that these curves are not congruent.
\end{proof}	

For the trefoil knot (the $(2,3)$-torus knot), the ropelength-minimal configuration seems to be $3$-fold symmetric, and hence one of the two configurations predicted by our Theorem. But in general, this is not true for torus knots: as we see in Figure~\ref{fig:52} for the $(2,5)$-torus knot, the minimal length configuration seems to break both symmetries. When $p$ and $q$ are not prime, there are also configurations with any symmetry dividing $p$ or $q$. If $a$ divides $q$, we cannot guarantee that the critical $a$-fold symmetric configuration is different from the $q$-fold symmetric configuration (which is also $a$-fold symmetric), but this is often the case. For example, we see in Figure~\ref{fig:29torus} configurations of the $(2,9)$ torus knot with $2$-fold, $3$-fold, and $9$-fold symmetry, as well as the (apparent) minimizer, tightened with no symmetry restrictions and breaking all three symmetries.
	
\begin{figure}[ht]

\begin{tabular}{ccc} 	
\includegraphics[height=1.1in]{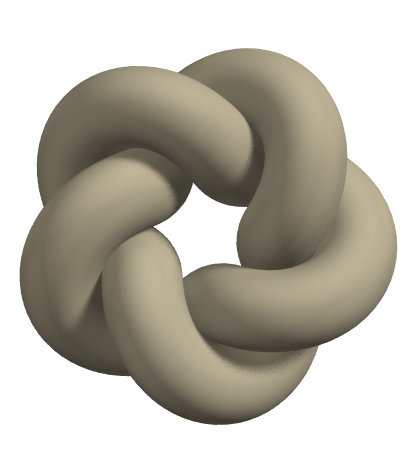} & 
\includegraphics[height=1.1in]{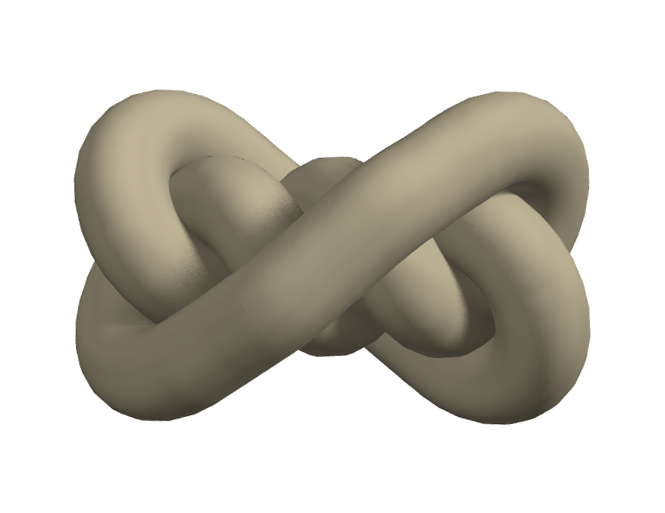} & 
\includegraphics[height=1.1in]{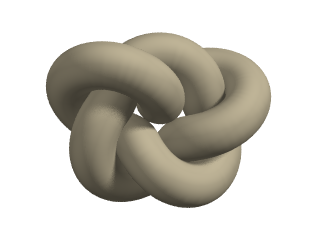} \\
$\Rop(L) \simeq 48.23$ &
$\Rop(L) \simeq 62.56$ &
$\Rop(L) \simeq 47.21$ \\
$\Z/5\Z$ & $\Z/2\Z$ & None \\
\end{tabular}

\caption{These figures show 5-fold and 2-fold symmetric configurations of the $(2,5)$-torus knot, as well as the overall ropelength minimizer, which appears to break both symmetries. The ropelengths for all three configurations are quite different.
\label{fig:52}}
\end{figure}	

\begin{figure}[ht]
\begin{tabular}{cccc} 	
 	\includegraphics[height=1.1in]{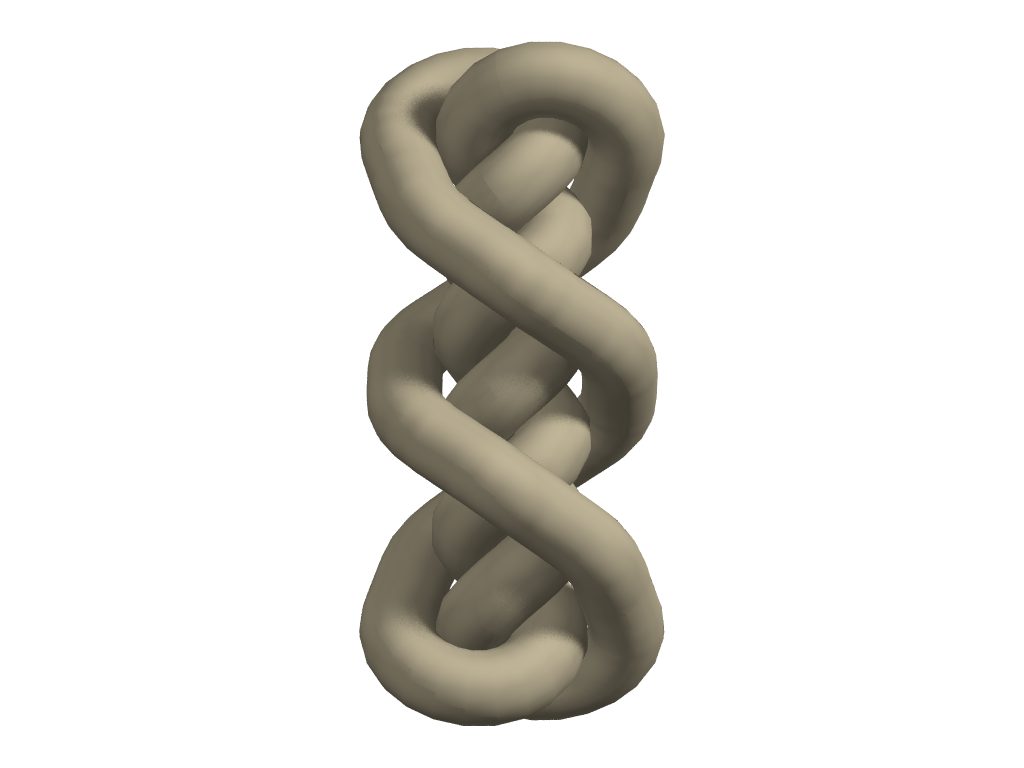} & 
	\includegraphics[height=1.1in]{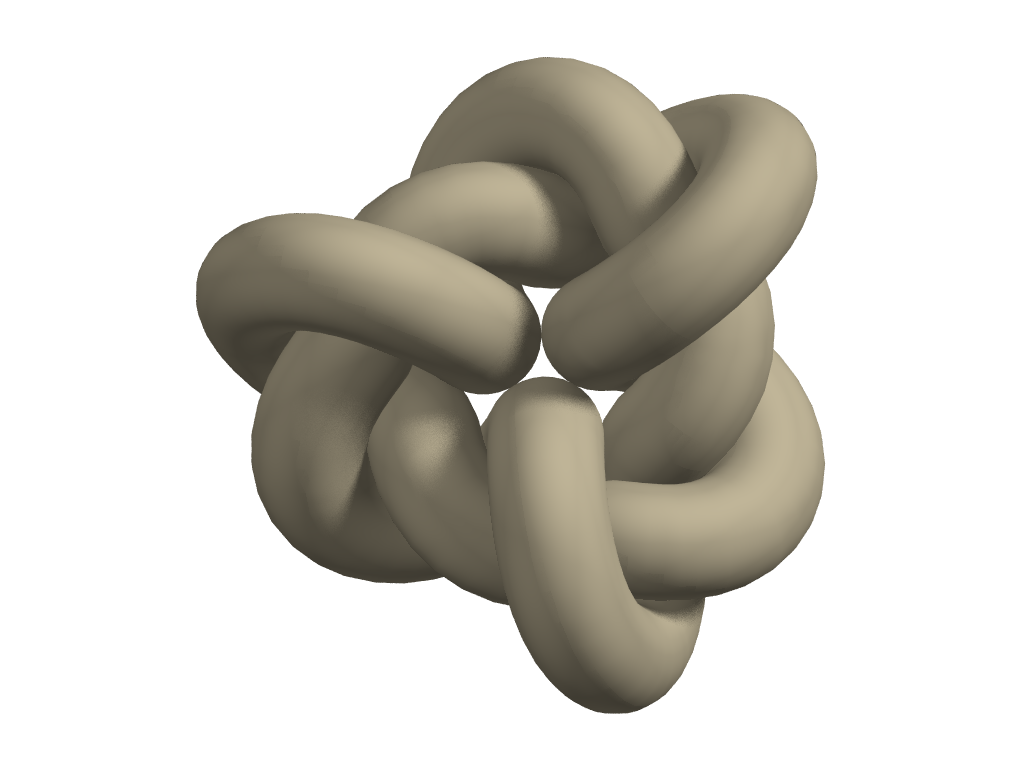} &
	\includegraphics[height=1.1in]{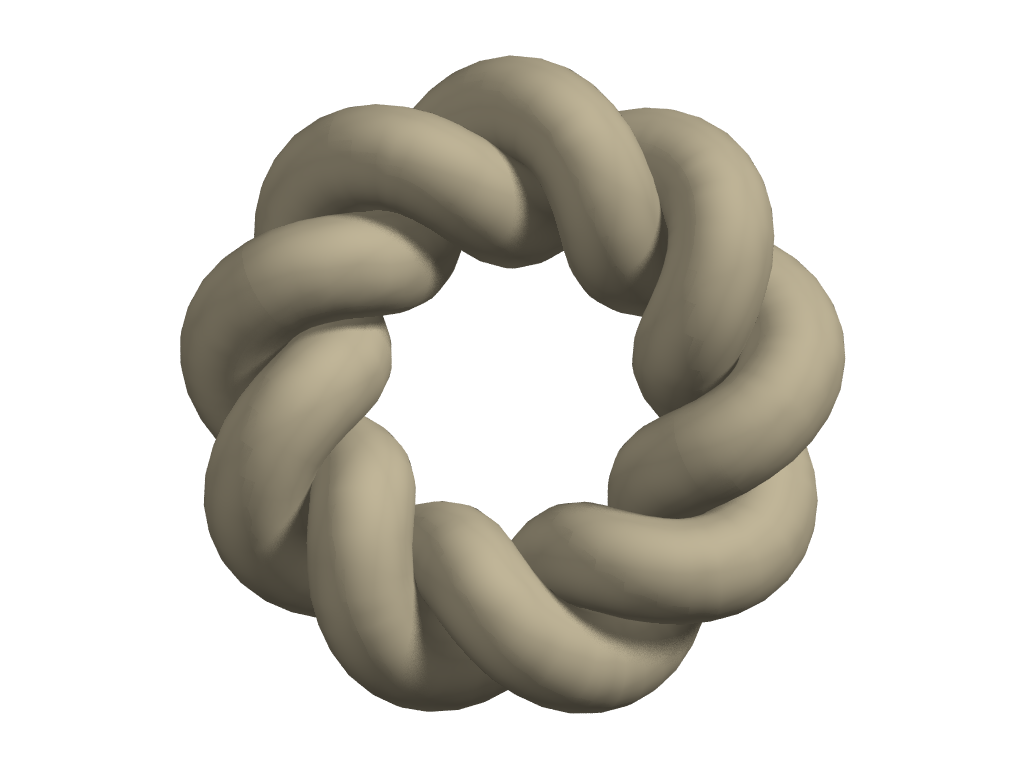} &
	\includegraphics[height=1.1in]{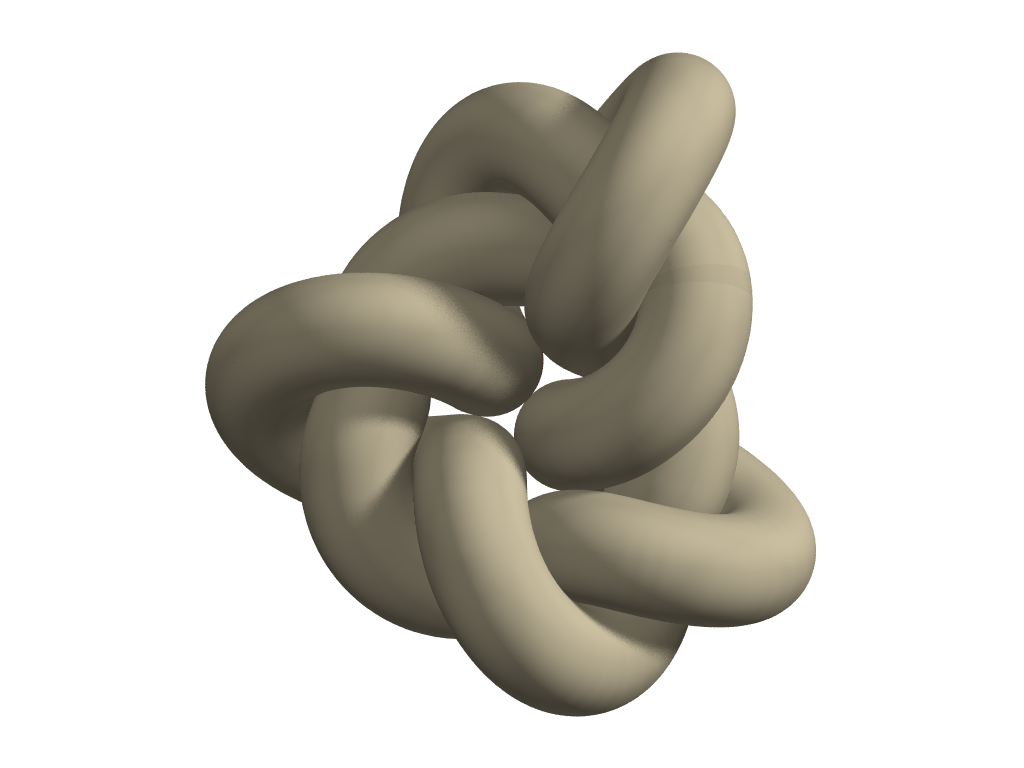} \\
	$\Rop(L) \simeq 114.23$ &	
	$\Rop(L) \simeq 75.91$ &		
	$\Rop(L) \simeq 80.68$ &
	$\Rop(L) \simeq 75.69$ \\
	$\Z/2\Z$ & $\Z/3\Z$ & $\Z/9\Z$ & None \\
\end{tabular}
\caption{Several critical configurations of the $(2,9)$ torus knot computed by \texttt{ridgerunner}\cite{Ashton:2011du}, with ropelengths and symmetry groups. We can see that the minimal length $3$-fold symmetric configuration is considerably different (and notably shorter) than the $9$-fold and $2$-fold symmetric minimizers. The overall minimizer (at least, according to this calculation) seems to break all of the symmetries of the torus knot, though it is only slightly shorter than the $3$-fold symmetric configuration which it resembles.
\label{fig:29torus}}
\end{figure} 

\subsection{Mirror Symmetric Knots and Links}

A version of the ropelength can be constructed defining a ``thick'' isotopy of a link to be an isotopy through curves of thickness at least one and a ``physical'' isotopy to be a thick isotopy which preserves the length of each component. In 2012, Coward and Hass \cite{Coward:2012wh} considered the case of a link constructed from composing two nontrivial knots so that the composite knot met the $x-y$ plane perpendicularly in exactly two places and adding a unknotted component around these intersections (see Figure~\ref{fig:splitlinks}). If the unknotted component wraps tightly around the two strands of the knot puncturing the plane, it has length $4\pi + 4$. Coward and Hass proved that during any thick isotopy which splits the link, the length of the unknotted component is at some point equal to $4 \pi + 6$, and hence that there is no physical isotopy which splits the link. 

We note that this theorem is almost, but not quite the same as proving that such a configuration admits a local minimum for ropelength in our sense. For us, the ropelength of a link is the sum of the lengths of the two components of the link divided by the thickness of the entire link. The ropelength of the knotted component of the link might be so much longer than its minimum ropelength in the configuration where the unknotted component has ropelength $4\pi + 4$ that the sum of the ropelengths of the two components manages to decrease during a splitting isotopy even as the length of the unknotted component increases.

It would be interesting to use the Coward/Hass methods to try to prove the existence of a ropelength local minimizer for these links. Our methods yield a weaker result; the existence of non-congruent ropelength critical configurations. First, it is easy to see that the connect sum of a knot and its mirror image has a mirror-symmetric configuration.

\begin{proposition}
For any knot or link $K$ the connect sum of $K$ and its mirror image $K \# K^m$ has a ropelength-critical configuration which is mirror-symmetric.
\label{prop:mirror}
\end{proposition}

\begin{proof}
It is easy to construct a mirror-symmetric representative of the link type $[K \# K^m]$ by reflecting a copy of $K$ over the plane and then joining the two components in a symmetric manner. The result then follows directly from Theorem~\ref{thm:master theorem}.
\end{proof}

We can construct examples similar to those of Coward and Hass from any mirror-symmetric knot, as shown in Figure~\ref{fig:splitlinks}.

\begin{figure}[ht]
\begin{tabular}{c@{\hspace{0.75in}}c}
\includegraphics[height=1.5in]{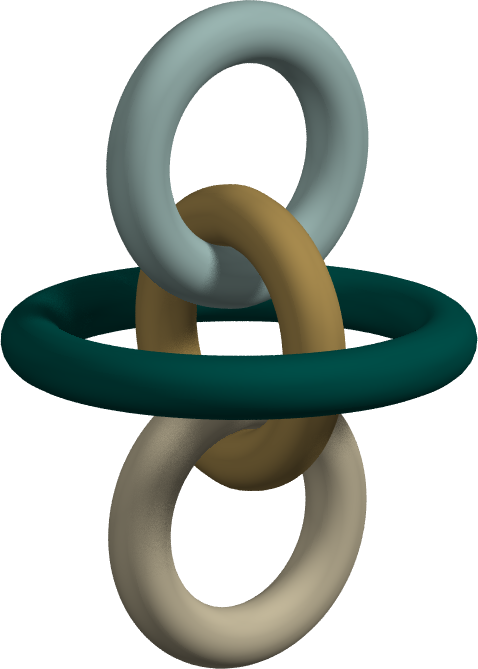} 
\includegraphics[height=1.5in]{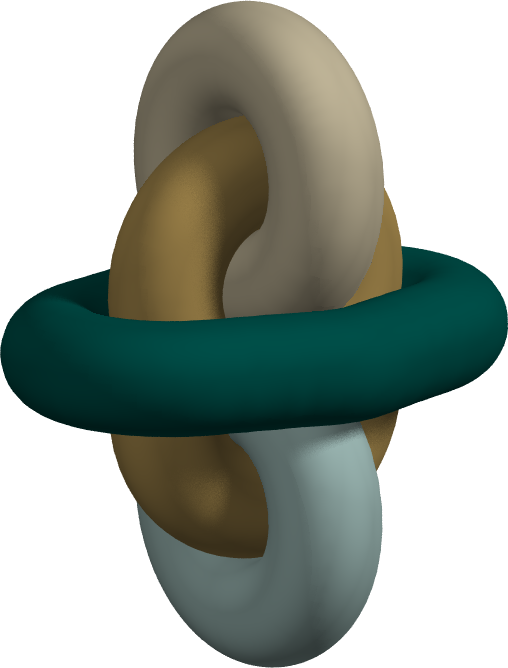} &
\includegraphics[height=1.5in]{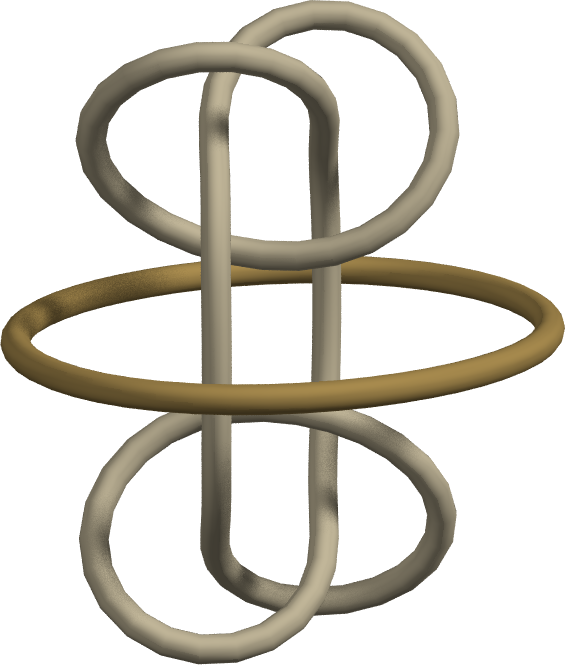}  
\includegraphics[height=1.5in]{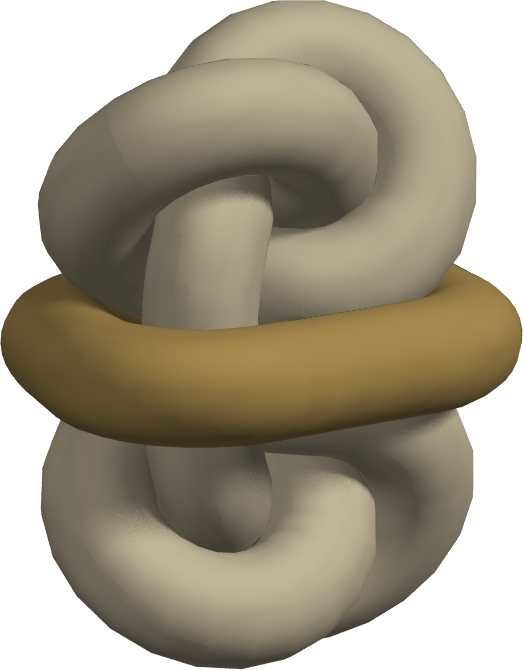} \\
$K_1 = 2^2_1 \# (2^2_1)^m$ (Three Link Chain) & $K_1 = 3_1 \# 3_1^m$ (Square Knot) 
\end{tabular}
\caption{This picture shows mirror-symmetric ropelength critical configurations which illustrate Proposition~\ref{prop:gordian}. On the left, we see the connect sum of a Hopf link and its mirror image forming a mirror-symmetric three link chain, together with a split component in the plane of symmetry. The second image shows an approximately tight configuration with this symmetry computed by \texttt{ridgerunner}\cite{Ashton:2011du}. On the right, we see the connect sum of a trefoil knot and its mirror image forming a mirror-symmetry square knot, together with a split component in the plane of symmetry. Coward and Hass~\cite{Coward:2012wh} show that the unknotted component cannot be split from the remainder of the link by an isotopy of curves with thickness $\geq 1$ without increasing its length.
\label{fig:splitlinks}}
\end{figure}

\begin{proposition}
Suppose that $[K]_{\Z/2\Z}$ is a symmetric link type where $\Z/2\Z$ acts by reflection over a plane. Let $[L]$ be the (topological) link type $K \sqcup O$, where $O$ is an unknot split from $K$. There is a symmetric link type $[L]_{\Z/2\Z} \subset [L]$ which contains a ropelength critical configuration which is not a global ropelength minimizer in $[L]$.
\label{prop:gordian}
\end{proposition}

\begin{proof}
Take any $K \in [K]_{\Z/2\Z}$. We can construct a representative of a symmetric link type $[L]_{\Z/2\Z}$ by placing the unknot $O$ in the plane of symmetry so that it encircles the intersections of $K$ with the plane; by Theorem~\ref{thm:master theorem}, $[L]_{\Z/2\Z}$ contains a symmetric, ropelength-critical configuration $L_0$. We must show that this configuration is not the ropelength minimizer in $[L]$. Now in every configuration in our link type $[L]_{\Z/2\Z}$, $O$ encircles the intersections of $K$ with the plane. After all, $O$ must remain planar to avoid breaking the symmetry of the component $O$, and $K$ cannot be made disjoint from the planar disk spanning $O$ without breaking the symmetry of $K$. Thus in our critical configuration $L$, $O$ encircles at least two intersections of $K$ with the symmetry plane, and hence the ropelength of $O$ is at least $4\pi + 4$ as $O$ encircles at least two disjoint unit disks in the plane (cf. Theorem~10 of~\cite{Cantarella:2002dx}). The ropelength of $K$ in this configuration is at least as large as the minimum ropelength $\Rop([K])$ of knot type $[K]$, and so $\Rop(L_0) \geq \Rop([K]) + 4\pi + 4$. However, the minimum ropelength of the topological link type $[L]$ is given by
\begin{equation*}
\Rop([L]) = \Rop([K]) + \Rop([O]) = \Rop([K]) + 2\pi,
\end{equation*}
since there is a split configuration in $[L]$ where $K$ and the unknot $O$ each achieve the minimum ropelength for their knot types.
\end{proof}
This construction gives us the first explicit examples of ropelength-critical configurations of links which are known not to be ropelength-minimal. It would be interesting to do this for knots.

\subsection{Sakuma's Examples}

When we examined the torus knots above, we considered the free periods of a knot. A related notion is the topological symmetry group of a knot or link $\Sym(L) = \MCG(S^3,L)$, where $\MCG(S^3,L)$ is the mapping class group of the pair $(S^3,L)$. A $G$-invariant configuration of $L$ induces an inclusion homomorphism $\iota \co G \subset \Isom \rightarrow \Sym(L)$. Further, any $G$-invariant isotopy of this configuration induces the same homomorphism $\iota$. Following Sakuma~\cite{MR1006701}, we can sometimes combine this construction with the group structure of $\Sym(L)$ to prove the existence of distinct configurations of $L$ with the \emph{same} type of symmetry! This is like the existence of different ``symmetric union diagrams'' of knots exhibited in~\cite{Eisermann:2007tq}, though those diagrams do not correspond to rigid symmetries of any space curve.

\begin{proposition}
Suppose that there exist configurations $L_1$ and $L_2$ in a nontrivial knot type $[L]$ which are $G_1$ and $G_2$ symmetric, and that the smallest subgroup of $\Sym(L)$ containing $\iota(G_1)$ and $\iota(G_2)$ is infinite. Then there exist distinct ropelength-critical configurations in $[L]$ which are $G_1$ and $G_2$ invariant (even if the groups $G_1$ and $G_2$ are isomorphic!).
\label{prop:sakuma style}
\end{proposition}

\begin{proof}
Let $[C_i]$ be the connected component of $G_i$-invariant representatives of $[L]$ containing $L_i$. We want to prove that $C_1$ and $C_2$ are disjoint. If not, there is a configuration $L^*$ in $[L]$ which is invariant under a symmetry group $H$ containing both $G_1$ and $G_2$ as subgroups which have only the identity element in common. Now $\iota(H)$ contains both $\iota(G_1)$ and $\iota(G_2)$ as subgroups, and hence $\iota(H)$ is infinite by hypothesis. This is a contradiction: since the knot $L^*$ is $H$-invariant and nontrivial, $H$ is finite by Proposition~\ref{prop:symknots}. 
\end{proof}

It then remains to construct knots with the required property. Here we rely again on Sakuma to show:

\begin{proposition}
Let $K$ be an untwisted double of the $6_1$ knot. Then $K$ has at least $4$ distinct $\Z/2\Z$-symmetric ropelength critical configurations.
\label{prop:distinct}
\end{proposition}

\begin{proof}
In Proposition 2.7~of~\cite{MR1006701}, Sakuma shows that for this knot there are four distinct $\Z/2\Z$ subgroups in $\Sym(L)$ with the property required by Proposition~\ref{prop:sakuma style}. Figure~\ref{fig:sakuma calculation} shows that we may construct symmetric representatives for each of these symmetry groups. 
\end{proof}

\begin{figure}[ht]
\hfill 
\includegraphics[height=1.5in]{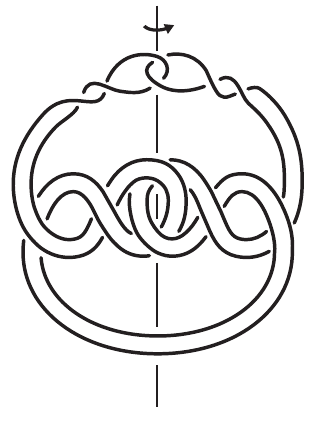} \hfill
\includegraphics[height=1.5in]{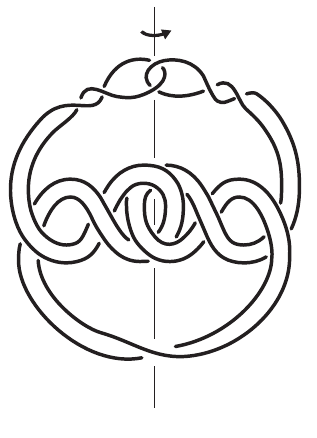} \hfill
\includegraphics[height=1.5in]{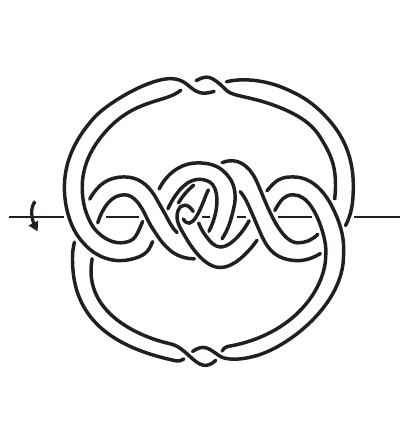} \hfill
\includegraphics[height=1.5in]{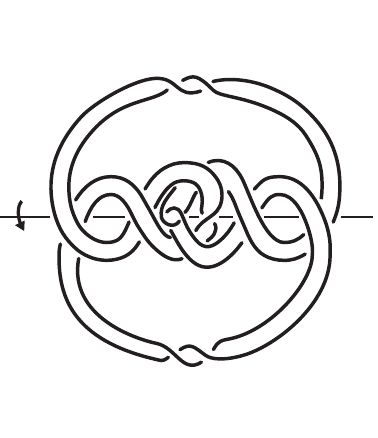} \hfill
\hphantom{.}
\caption{This picture shows four different configurations of the untwisted double $K$ of the twist knot $6_1$, each exhibiting a different subgroup of  $\Sym(K)$ isomorphic to $\Z/2\Z$. Sakuma~\cite{MR1006701} shows that these subgroups have the interesting property that any two generators together generate an infinite subgroup of $\Sym(K)$. This implies by Proposition~\ref{prop:sakuma style} that there are at least four distinct ropelength-critical configurations of $K$ which are $\Z/2\Z$ invariant.
\label{fig:sakuma calculation}}
\end{figure}

\section{Future Directions}

We have extended symmetric criticality, a standard tool in the analysis of geometric functionals, to the study of ropelength and given a few examples of how this principle may be used to construct various ropelength-critical symmetric knots and links. It is clear that there are many more constructions of this kind available. For instance, we have only shown examples of ropelength-critical knots with cyclic and dihedral symmetry, since these are the only symmetries possible for a knot. For links, it would be interesting to construct and numerically tighten links exhibiting the other point groups in $\Orth$. For instance, the ropelength-critical Borromean rings in~\cite{Cantarella:2011vq} has order 24 pyritohedral symmetry. We note that while our paper~\cite{Cantarella:2011vq} also covered the case of open curves with various endpoint constraints, we did not consider those cases here. We do not expect there to be any essential difficulty in carrying out such an extension, only some additional technical details to keep track of.

When considering the mirror-symmetric composite knots, we proved that a composite in the form $K \# K^m$ has a ropelength-critical configuration which is mirror-symmetric over a plane. It would be interesting to show that this configuration must have two straight segments where it crosses the plane, using an analysis similar to the methods in~\cite{Cantarella:2011vq}. 

Lastly, it would be very interesting to extend other general tools from the analysis of geometric functionals to this setting. A natural candidate would be the ``mountain pass principle''~\cite{MR2012778}: Given a smooth functional $f \co X \rightarrow \R$, whenever one has two points $x$ and $x'$ so that every path in $X$ from $x$ to $x'$ contains a point $x''$ where $f(x'') > f(x)$ and $f(x'') > f(x')$, there exists an unstable critical point of $f$ along a path joining $x$ and $x'$. The theorem of Coward and Hass is expected to be valuable in this context, as it gives a concrete example of a pair of configurations $L$ and $L'$ of a link for which such conditions could be established for ropelength.

\section{Acknowledgements}

The idea of finding symmetric configurations of knots which are critical for various energies is surely quite classical, and we must be among many mathematicians who have thought about this question. We have not attempted to give a comprehensive history of related work. We were first introduced to the idea of finding different critical configurations of the trefoil using symmetry by Moffatt's 1990 paper. Our intuition was sharpened by productive conversations with many colleagues over the intervening years, especially an intermittent series of conversations over the past decade with Rob Kusner and John Sullivan. We would also like to thank Ryan Budney. It is worth noting that while the work above to extract the symmetric criticality theorem is fairly elementary, this is only because the difficult part of the theorem is proved in~\cite{Cantarella:2011vq}: without those methods, these questions would have remained quite challenging to approach. Cantarella would also like to acknowledge the gracious hospitality of the Issac Newton Institute in Fall 2012.

\newpage

\bibliographystyle{alpha}
\bibliography{symmetric,papers}

\end{document}